\documentclass[11pt,a4paper]{elsarticle} 
\usepackage[latin1]{inputenc}
\usepackage[english]{babel}
\usepackage{amsmath,amssymb,amsthm,bm}
\usepackage{cases}
\usepackage{enumerate}
\usepackage{color}
\usepackage{comment}
\theoremstyle{plain} 
\newtheorem{theorem}{Theorem}
\newtheorem*{theorem*}{Theorem}
\newtheorem{proposition}[theorem]{Proposition}
\newtheorem{lemma}[theorem]{Lemma}
\newtheorem{corollary}[theorem]{Corollary}

\newtheorem{remark}[theorem]{Remark}

\usepackage{amssymb}
\usepackage{comment} 
\usepackage{dsfont}
\def\mand{\qquad\mbox{and}\qquad}
\usepackage{indentfirst}
\usepackage[margin=1in,footskip=0.25in]{geometry}

\def\ssum{\mathop{\sum\ldots \sum}}

\newcommand{\re}{\textup{Re}}
\newcommand{\im}{\textup{Im}}
\newcommand{\abs}[1]{\left\lvert#1\right\rvert}

\allowdisplaybreaks[3]

\title{\bf Bounds for moments of quadratic character sums and theta functions}

\author{
  Marc Munsch, Yuichiro Toma}
  
\newcommand{\Addresses}{{
  \bigskip
  \footnotesize

   \textsc{Marc Munsch,  Institut Camille Jordan, Universit\'{e} Jean Monnet
42000 St-Etienne, 20 Rue du Dr R\'{e}my Annino, France}\par\nopagebreak
  \textit{E-mail address:} \texttt{marc.munsch@univ-st-etienne.fr}

  \textsc{Yuichiro Toma,  Global Education Center, Waseda University, 1-6-1 Nishiwaseda, Shinjuku-ku, Tokyo 169-8050, Japan}\par\nopagebreak
  \textit{E-mail address:} \texttt{yuichiro.toma@aoni.waseda.jp}

}}

\begin{document}

\begin{abstract}
In this paper, we investigate the size of moments of quadratic character sums averaged over the family of fundamental discriminants. We obtain an asymptotic formula for all integer moments in a restricted range of parameters using a multivariate tauberian theorem. As a consequence, we prove unconditional lower bounds for all even integer moments of quadratic character sums in a wide range of parameters. Moreover, assuming the Generalised Riemann Hypothesis (GRH), we prove a sharp upper bound on moments of character sums of arbitrary length. In a similar fashion, we obtain unconditional lower bounds on moments of quadratic theta functions and matching conditional upper bounds under GRH. In the case of the second moment of theta functions, we prove an optimal upper bound unconditionally improving the previous results of Louboutin and the first named author.
\end{abstract}

\maketitle

\section{Introduction}
A fundamental problem in analytic number theory is to give precise estimates for 
$$\sum_{n \leq x} \chi(n) $$ where $\chi$ is a non-principal Dirichlet character modulo $q$. The ultimate goal would be to obtain cancellation for $x$ as small as possible. In many instances, it is sufficient to have precise information about the means of character sums, namely $$M_{k,q}(x):=\frac{1}{q-1}  \sum_{\chi \bmod q \atop \chi \neq \chi_0 } \left\vert \sum_{n \leq x} \chi(n) \right\vert^k, \qquad (k>0).$$ 
A lot of effort has been put to understand the size of $M_{k,q}(x)$ for the widest possible range of $x$. In the case of \textit{even integer moments}, precise results are known (see  \cite{ACZ,CZ,GS,KE,MV,thetalow,SZ1} and the references therein).  Recently, Harper \cite{H4} initiated the study of \textit{low moments} and obtained surprising upper bounds that beat the usual expected square-root cancellation. These results are related to his breakthrough work \cite{H2} on the problem of low moments of Steinhaus random multiplicative functions. In this paper, we focus on the important family of real characters.  \\

We define for $k>0$ the moments of quadratic character sums as follows:

$$S_{k}(X,Y):= \sideset{}{^*}{\sum}_{0<d \leq X}  \left( \sum_{n \leq Y} \chi_{d}(n) \right)^{k}$$  where $\sum^*$ indicates a sum over fundamental discriminants. In relation to class number problems, Jutila \cite{Ju1,Ju2} initiated the study of $S_{k}(X,Y)$  and formulated the following guiding conjecture. For any even integer $k\geq 2$ and large $X,Y$, there exist constants $c_1(k)$ and $c_2(k)$ such that 

$$ S_k(X,Y) \leq c_1(k)XY^{k/2} (\log X)^{c_2(k)}.$$ Unconditionally, this conjecture is known to be true for $k=2$ (with $c_2(2)=1$, see \cite{Armon,Ju1}). For $k=4$, using Heath-Brown's large sieve \cite{HB}, Virtanen \cite{Vir} managed to get a partial result replacing the power of logarithm by $X^{\varepsilon}$. Recently, Gao and Zhao \cite[Theorem $1.3$]{GZ} proved under the assumption of the Generalized Riemann Hypothesis (GRH) that for any real $k>1+\sqrt{5}$ and large $X,Y$, \begin{equation}\label{upperGZ}S_k(X,Y) \ll XY^{k/2}(\log X)^{\frac{k(k-1)}{2}+1}.\end{equation} Note that a similar result was obtained in \cite{ZZ} considering the family of prime discriminants. Their approach relied on optimal bounds on shifted quadratic $L$-functions, a strategy that appeared previously in the work of the first named author \cite{Munschtheta} and Szab\'{o} \cite{SZ2} in the case of $L$-functions associated to characters of fixed modulus. \\ 
Our goal in this note is to refine these results by obtaining lower and upper bounds of the same order of magnitude in the case of \textit{high} integer moments. We prove the following lower bound for all integer moments in a wide range of parameters.

\begin{theorem}\label{lowquad} For any fixed even integer $k\geq 2$ and $X^{\varepsilon} \ll Y\ll X^{2/3-\alpha}$ for some $\alpha>0$,  we have
 
\begin{equation}\label{lowSk} S_{k}(X,Y) \gg XY^{k/2} (\log X)^{\frac{k(k-1)}{2}}  \end{equation}  Moreover, assuming GRH, \eqref{lowSk} remains valid in the range  $X^{\varepsilon} \ll Y\ll X^{1-\alpha} $ for some $\alpha>0$.
\end{theorem}In the case of moments of zeta and $L$-functions, several methods have been developed to get lower bounds for non-integer values of the exponent $k$. Without being exhaustive, Radziwill and Soundararajan \cite{RaS} obtained optimal results for $k \geq 2$ while the recent work of Heap and Soundararajan \cite{HS} deals with the case of smaller $k$ (see also the recent results \cite{Curran,GZ3} developing this approach in the case of shifted moments). In the case of character sums modulo a large fixed prime $q$, based on the work of Harper  \cite{H1,H3,H4}, Szab\'{o} \cite{SZ2}  managed to prove lower bounds of the expected size for real $k\geq 2$. We leave this open for future work in the case of quadratic characters. Moreover, Harper \cite{H2,H4} highlighted the fact that a different phenomenon occurs for \textit{low} moments in the family of Dirichlet characters modulo $q$ which might be of interest to study in our situation. The best known lower bounds (losing a small power of logarithm) were obtained by the first named author in collaboration with de la Bret\`{e}che and Tenenbaum \cite{dlbMT}. \\

If $Y$ is very small compared to $X$, we can prove an asymptotic formula for $S_k(X,Y)$. To state the result, we introduce $r=k(k-1)/2$ linear forms $\mathcal{H} := \left( h^{(1)}, \dots, h^{(r)}\right)$ from $\mathbb{C}^k$ to $\mathbb{C}$ whose restriction to $\mathbb{R}_{\geq 0}$ have values in $\mathbb{R}$ as follows. For $\bm{x}=(x_1,\dots,x_k) \in \mathbb{C}^k$, we define 
\begin{align*}
\left( h^{(1)}(\bm{x}), \dots, h^{(r)}(\bm{x})\right) &= (x_1 \ x_2 \ \cdots \ x_k) \left( \bm{a}_1 \ \bm{a}_2 \ \cdots \ \bm{a}_r \right), 
\end{align*}
where $\bm{a}_1, \bm{a}_2, \dots, \bm{a}_r \in M_{k,1}(\mathbb{N}_0)$ are column vectors of length $k$ with two entries equal to $1$ and all other entries equal to $0$. Then, we obtain the following result.

\begin{theorem}
\label{asymptotic for S_k(X,Y)}
For any fixed integer $k\geq 1$ and  $Y \ll X^{\frac{2}{3k}}(\log X)^{\frac{2}{3}(k-1)-\frac{4}{3k}}$, we have
\begin{align*}
S_k(X,Y) &= \frac{c_k \gamma_k}{\zeta(2)} XY^\frac{k}{2}(\log Y)^{\frac{k(k-1)}{2}}+ O\left(XY^\frac{k}{2}(\log Y)^{\frac{k(k-1)}{2}-1}\right),
\end{align*}
where $c_k$ is given by the Euler product
\begin{equation*}
  c_k=  \prod_{p} \frac{\left(1- \frac{1}{p} \right)^{\frac{k(k+1)}{2}}}{1+\frac{1}{p}} \left(\frac{1}{p}+\frac{1}{2}\left[\left(1-\frac{1}{\sqrt{p}}\right)^{-k}+\left(1+\frac{1}{\sqrt{p}}\right)^{-k} \right]\right) 
\end{equation*}
and $\gamma_k$ is the volume of a polytope in $\mathbb{R}^{r}$  defined as the set of $(u_i) \in \mathbb{R}^{r}$ such that,
$$ \textrm{for each }   j \leq k: \sum_{i=1}^{r} h^{(i)}(\bm{e}_j) u_i = 1.$$ In other words, we have 
\begin{align*}
    \gamma_k &= \int_{\mathcal{A}_k} du_1\cdots du_r,
\end{align*}
where
\[
\mathcal{A}_k = \left\{ (u_1,\dots,u_r) \in [0,1]^r \middle| \sum_{i=1}^{r} h^{(i)}(\bm{e}_j) u_i \leq 1 \text{ for all } 1 \leq j \leq k \right\}.
\]
\end{theorem}

\begin{remark}
    The geometric constant $\gamma_k$ can be computed for small values of $k$ (for instance $\gamma_2=1, \gamma_3=1/4$). However, we have not been able to give a closed formula for all values of $k$. This could be compared with the asymptotic of moments of Steinhaus random multiplicative functions proved in \cite[Theorem $3$]{HNR} where the volume of the famous Birkhoff polytope is involved in the geometric constant. 
\end{remark}



Assuming GRH, we provide sharp upper bounds on smoothed character sums, improving the results of \cite{GZ,GZ2} for positive integers $k$. For the sake of simplicity, we consider the smoothed moments:

$$S_k(X,Y,W):= \sideset{}{^*}{\sum}_{0<d \leq X} \left\vert \sum_{ n\geq 1} \chi_d(n)W(n/Y) \right\vert^k $$ where $W$ is any non-negative, smooth function compactly supported on the set of positive real numbers. Then, we have the following result.
\begin{theorem}\label{th_smooth}
Assume GRH. For any integer $k\geq 2$ and large $X,Y$ we have
\begin{equation*}S_k(X,Y,W) \ll XY^{k/2} (\log X)^{\frac{k(k-1)}{2 }}.\end{equation*}
\end{theorem}
 \begin{remark}\label{remark} Note that unconditionally, for any reals $k>0$ and $Y \geq X$, by \cite[Theorem $1$]{Armon},  we have 
$$S_k(X,Y) \ll X^{1+k/2} \ll XY^{k/2} (\log X)^{\frac{k(k-1)}{2}}.$$ 
Theorem \ref{th_smooth} mproves on \eqref{upperGZ} and is optimal in view of Theorem \ref{lowquad}. Moreover, it is in accordance with the case $k=2$ which was unconditionally proved by Armon \cite[Theorem $2$]{Armon}. Note that \eqref{upperGZ} gives a weaker exponent of the logarithm but is valid for real values of $k$. In order to avoid too much technicality, we focus on the case of integer moments. 
 \end{remark}

For an even primitive Dirichlet character $\chi$ modulo $f$, let 
$$\theta(t,\chi):=\sum_{n\geq 1} \chi(n)e^{-\frac{\pi n^2 t}{f}} \qquad (t>0)$$ be its associated theta function. The study of moments of $\theta(1,\chi)$ averaged over the unitary family of primitive characters modulo a fixed prime $q$ has been extensively studied in connection with a conjecture of Louboutin \cite{Lou} about non-vanishing of $\theta(1,\chi)$  (see results in \cite{dlbMT,LM1} towards this conjecture). The behavior of moments is similar to the case of character sums of length $\approx \sqrt{q}$.  For a real $k\geq 2$, the following bounds hold.
\begin{equation}\label{thetaXq} X^{1+k/2} (\log X)^{(k-1)^2} \ll \sum_{\chi \bmod q \atop \chi(-1)=1} \left \vert \theta(1,\chi)\right\vert^{2k} \ll X^{1+k/2} (\log X)^{(k-1)^2} \end{equation} where the lower bound \cite{thetalow,SZ2} is unconditional and the upper bound \cite{Munschtheta,SZ1} requires the assumption of GRH. We pursue this analogy in the case of theta functions averaged over fundamental discriminants. For $k>0$ we consider the moments of real theta functions:
$$ \sideset{}{^*}{\sum}_{0<d\leq X} \vert \theta(1,\chi_d)\vert^{k}.$$ 
We prove unconditional lower bounds of the right order of magnitude which may be compared to the ones for $L$-functions of Rudnick and Soundararajan \cite{RS1,RS2} and to the ones obtained by the first named author and Shparlinski \cite{thetalow}. 
\begin{theorem}\label{lowtheta} For any fixed even integer $k\geq 2$, we have $$ \sideset{}{^*}{\sum}_{0<d\leq X} \vert \theta(1,\chi_d)\vert^{k} \gg X^{1+k/4}(\log X)^{\frac{k(k-1)}{2}}.  $$ 
\end{theorem} 
Moreover, under GRH, we can argue as in the proof of Theorem \ref{th_smooth} and obtain sharp upper bounds similar to the results of Szab\'{o} \cite{SZ1}.
\begin{theorem}\label{th_theta}
Assume GRH. For any fixed integer $k\geq 2$ and large $X,Y$ we have
\begin{equation*} \sideset{}{^*}{\sum}_{0<d\leq X} \vert \theta(1,\chi_d)\vert^k \ll X^{1+k/4} (\log X)^{\frac{k(k-1)}{2 }}.\end{equation*}
\end{theorem}

\begin{remark} This is optimal in view of Theorem \ref{lowtheta}.  It is worth mentioning that for \textit{low} moments ($k\leq 2$) over even characters modulo $q$, Harper \cite{H4} recently obtained  conjecturally sharp upper bounds. The problem of proving matching lower bounds remains open and only partial results are known \cite{dlbMT}.\end{remark}
We unconditionally obtain the order of magnitude for the second moment of theta functions. This should be compared to the aforementioned result of Armon \cite[Theorem $2$]{Armon} in the case of character sums. It was previously known by \cite[Theorem $2$]{LM} and \cite{Munsch} that $$\sideset{}{^*}{\sum}_{0<d\leq X} \vert \theta(1,\chi_d)\vert^2 \ll X^{3/2} (\log X)^{2}.$$ 

\begin{theorem}\label{theta2nd} For $X$ large, we have
$$\sideset{}{^*}{\sum}_{0<d\leq X} \vert \theta(1,\chi_d)\vert^2 \asymp X^{3/2}(\log X).  $$ 
\end{theorem}
A standard application of Cauchy-Schwarz inequality leads to the following non-vanishing result improving \cite[Theorem $2$]{LM}.
\begin{corollary}
There exist at least $\gg X/\log X$ fundamental discriminants $d\leq X$ such that $\theta(1,\chi_d) \neq 0$.
\end{corollary} Note that the method allows to prove the same result at any point $t_0>0$ and in arithmetic progressions as in \cite[Theorem $2$]{LM}. In the family of even characters modulo a fixed modulus, stronger results are known by \cite[Theorem $1.4$]{dlbMT} using a combinatorial argument involving G\'{a}l sums to construct suitable \textit{mollifiers}.

\section{Overview of the proofs}
\par In order to obtain lower bounds for moments of character sums, we follow the method of Rudnick and Soundararajan  \cite{RS1,RS2} which was developed in the case of $L$-functions at the central point. A similar strategy was employed by the first named author and Shparlinski~\cite{thetalow} in order to obtain a lower bound on the moments of theta functions averaged over Dirichlet characters modulo $q$. This required an asymptotic for the number of solutions to the equation $x_1\cdots x_k= x_{k+1} \cdots x_{2k}$ where $x=(x_1,\dots,x_k)$ lies in a rectangular box. In our case, we need a result about the number of weigthed solutions to the equation $x_1\cdots x_k = \square$ where $x=(x_1,\dots,x_k)$ lies in a rectangular box. Both equations can be treated using de la Bret\`{e}che multivariate tauberian theorem \cite{dlB}. 
Note that a similar count was treated in \cite{KS}. \\

 For the upper bounds, we use Mellin inversion to bound the moments of character sums by the average of shifted moments of $L$-functions on the critical line. This method was introduced by the first named author in \cite{Munschtheta} for the moments of theta functions over the subgroup $X_q^+$ of even characters modulo a fixed prime $q$. This required optimal and uniform bounds on shifted moments of $L$-functions as obtained by the first named author and refined by Szab\'{o} \cite{SZ1}. These estimates were obtained by adapting the work of Soundararajan \cite{Sound} on conditional upper bounds for the moments of the Riemann zeta function, sharpened by Harper \cite{H1} for the optimal bounds. In the case of quadratic characters, following this strategy, Gao and Zhao proved an optimal result in \cite{GZ} (see Proposition \ref{th_shifted} below). Applying this result, it remains to optimally estimate the resulting multiple integral involving the shifts. We proceed differently as in \cite{GZ} using a new induction argument.  \\
 

\section{Background results}

We record several technical results. The first is the following ``orthogonality relation'' for quadratic characters, which is a simple application of the P\'olya-Vinogradov inequality.
\begin{lemma}\label{sqrfree}
For all positive integers $n$ we have 
$$ \sideset{}{^*}{\sum}_{0<d \leq X} \chi_d(n)\ll X^{1/2}n^{1/4}\log n,$$
if $n$ is not a perfect square. On the other hand if $n=m^2,$ then 
$$ \sideset{}{^*}{\sum}_{0<d \leq X} \chi_d(n)=\frac{6}{\pi^2}X\prod_{p|m}\left(\frac{p}{p+1}\right)+ O\big(X^{1/2} \tau(m)\big).$$ Under GRH, we can write it in a more compact way. For any $\varepsilon>0$, 
	\begin{align*}
	\sideset{}{^*}{\sum}_{0<d \leq X} \chi_{d}(n)=\frac{6}{\pi^2}X\prod_{p|n}\left(\frac{p}{p+1}\right)\mathds{1}_{n=\square}+ O\left(X^{1/2+\varepsilon}n^{\varepsilon}\right),
	\end{align*}
where $\mathds{1}_{n=\square}$ indicates the indicator function of the square numbers. 
\end{lemma}



\begin{proof}
The unconditional result is \cite[Lemma $4.1$]{GrSo} while the conditional one follows from \cite[Lemma $1$]{DM}. 
\end{proof}

The following lemma follows from partial summation and allows us to deal with weighted character sums. \begin{lemma}\cite[Lemma $14$]{Armon}\label{lem:weight}
Suppose that $k\leq l$ and that $a_n$ and $b_n$ are complex numbers. Let
$$S(n)= \sum_{m=k}^{n}a_m.$$
Let $M=\max_{k\leq n\leq l} \vert b_n \vert$ and $V=\sum_{n=k}^{l-1} \vert b_n-b_{n+1}\vert.$ Then
$$ \left\vert \sum_{n=k}^{l} a_n b_n \right\vert \leq (M+V) \max_{k\leq n\leq l}S_n.$$
\end{lemma} 

The next two results due to de la Bret\`{e}che allow us to obtain asymptotic estimates for partial sums of multivariate arithmetic functions. We denote the canonical basis of $\mathbb{C}^k$ by $\{\bm{e}_j\}_{j=1}^k$ and its dual basis by $\{\bm{e}_j^*\}_{j=1}^k$.

\begin{lemma}[{\cite[Th\'{e}or\`{e}me 1]{dlB}}]
\label{thm:breteche1}
    Let $f: \mathbb{N}^k \to \mathbb{R}$ be a non-negative function and $F$ the associated Dirichlet series of $
f$ defined by
\[
F(\bm{s}) = F(s_1,\dots,s_k) = \sum_{n_1,\dots,n_k=1}^\infty \frac{f(n_1,\dots,n_k)}{n^{s_1}\cdots n_k^{s_k}}.
\]
Denote by $\mathcal{LR}_k^+(\mathbb{C})$ the set of non-negative $\mathbb{C}$ linear forms from $\mathbb{C}^k$ to $\mathbb{C}$ on $(\mathbb{R}_{\geq 0})^k$. Moreover, assume that there exists $(c_1, \dots,c_k) \in (\mathbb{R}_{\geq 0})^k$ such that:
\begin{enumerate}[(1)]
\item For $\bm{s} \in \mathbb{C}^k$, $F(s_1,\dots,s_k)$ is absolutely convergent for $\re(s_i) > c_i$ for all $1 \leq i \leq k$.

\item There exist a finite family $\mathcal{L} = (l^{(i)})_{1\leq i \leq q}$ of non-zero elements of $\mathcal{LR}_k^+(\mathbb{C})$, a finite family $(h^{
(i)})_{1\leq i \leq q^\prime}$ of elements of $\mathcal{LR}_k^+(\mathbb{C})$ and $\delta_1, \delta_3>0$ such that the function $H$ defined by
\[
H(\bm{s}) = F(\bm{s} + \bm{c}) \prod_{i=1}^q l^{(i)}(\bm{s})
\]
has a holomorphic continuation to the domain
\begin{align*}
D(\delta_1, \delta_3) 
&= \left\{ \bm{s} \in \mathbb{C}^k \mid \re\left(l^{(i)}(\bm{s})\right) > -\delta_1 \text{ for all } 1 \leq i \leq q \right. \\
& \qquad\qquad\qquad \left. \text{ and } \re\left(h^{(i)}(\bm{s})\right)> -\delta_3 \text{ for all } 1\leq i \leq q^\prime \right\}.
\end{align*}
\item There exist $\delta_2>0$ such that for $\varepsilon,\varepsilon^\prime>0$ we have uniformly in $\bm{s} \in D(\delta_1 -\varepsilon^\prime,\delta_3-\varepsilon^\prime)$ 
\[
H(\bm{s}) \ll \prod_{i=1}^q \left(\abs{\im\left(l^{(i)}(\bm{s})\right)}+1 \right)^{1-\delta_2 \min \left\{0, \re\left( l^{(i)}(\bm{s})\right)\right\}}(1+(\abs{\im(s_1)}+\dots+\abs{\im(s_k)})^\varepsilon).
\]
\end{enumerate}
Set $J:=J(\bm{c}) = \{j \in \{1, \dots, k\} \mid c_j = 0\}$. Denote $w$ as the cardinality of $J$ and by $j_1 <\dots<j_w$ its elements in increasing order. Define the $w$ linear forms $l^{(q+i)}$ $(1 \leq i \leq w)$ by $l^{(q+i)}(\bm{s}) = {e^{*}}_{j_i}(\bm{s}) = s_{j_i}$.

Then, for any $\bm{\beta} = (\beta_1,\dots, \beta_k) \in (0,\infty)^k$, there exist a polynomial $Q_{\bm{\beta}} \in \mathbb{R}[X]$ of degree at most $q + w - Rank \left(l^{(1)}, \dots, l^{(q+w)}\right)$ and $\theta > 0$ such that as $x \to \infty$
\[
\sum_{n_1\leq x^{\beta_1}}\dots\sum_{n_k \leq x^{\beta_k}} f(n_1,\dots,n_k) = x^{\langle \bm{c}, \bm{\beta} \rangle} Q_{\bm{\beta}}(\log x) +O\left( x^{\langle \bm{c}, \bm{\beta} \rangle - \theta}\right).
\]
Here, $\langle \cdot, \cdot \rangle$ denotes the usual dot product in $\mathbb{R}^k$.
\end{lemma}

The next theorem determines the precise degree and leading coefficient of the polynomial $Q_{\bm{\beta}}$ appearing in the previous theorem. We denote by $\mathbb{R}_*^+$ the set of strictly positive real numbers, the notation $con^*(\{l^{(1)},\dots,l^{(q)}\})$ means $\mathbb{R}_*^+ l^{(1)}+\dots+\mathbb{R}_*^+ l^{(q)}$.

\begin{lemma}[{\cite[(iii) and (iv) of Th\'{e}or\`{e}me 2]{dlB}}]
\label{thm:breteche2}
Let $f : \mathbb{N}^k \to \mathbb{R}$ be a non-negative function satisfying the assumptions of Lemma \ref{thm:breteche1}. Let $\bm{\beta} = (\beta_1,\dots,\beta_k) \in (0,\infty)^k$. Set $\mathcal{B} = \sum_{i=1}^k\beta_i \bm{e}_i^* \in \mathcal{LR}_k^+(\mathbb{C})$.
\begin{enumerate}[(1)]
\item If the Dirichlet series $F$ satisfies the additional two assumptions:
\begin{enumerate}[(C1)]
\item\label{C1}  There exists a function $G$ such that $H(\bm{s}) = G\left(l^{(1)}(\bm{s}),\dots,l^{(q+w)}(\bm{s}) \right)$.
\item\label{C2} $\mathcal{B} \in Vect \left(\{l^{(i)} \mid i = 1,\dots q + w\}\right)$ and there is no subfamily $\mathcal{L}^\prime$ of $\mathcal{L}_0 :=\left( l^{(i)}\right)_{1\leq i \leq q+w}$ such that $\mathcal{L}^\prime \neq \mathcal{L}_0$, $\mathcal{B} \in Vect(\mathcal{L}^\prime)$ and $\#\mathcal{L}^\prime -Rank(\mathcal{L}^\prime) = \#\mathcal{L}_0 - Rank(\mathcal{L}_0)$.
\end{enumerate}
Then, the polynomial $Q_{\bm{\beta}}$ satisfies the relation
\[
Q_{\bm{\beta}}(\log x) = H(\bm{0})x^{-\langle \bm{c}, \bm{\beta} \rangle}\mathcal{I}_{\bm{\beta}}(x)+O((\log x)^{\rho-1}),
\]
where $\rho := q+w-Rank(l^{(1)},\dots, l^{(q+w)})$ and 
\[
\mathcal{I}_{\bm{\beta}}(x) := \int_{\mathcal{A}_{\bm{\beta}}(x)} \frac{dy_1\cdots dy_q}{\prod_{i=1}^q y_i^{1-l^{(i)}(\bm{c})}},
\]
with
\[
\mathcal{A}_{\bm{\beta}(x)} := \left\{ \bm{y} \in [1,\infty)^q \middle| \prod_{i=1}^q y_i^{l^{(i)}(\bm{e}_j)} \leq x^{\beta_j} \text{ for all } 1 \leq j \leq k \right\}.
\]
\item If $Rank\left(l^{(1)},\dots,l^{(q+w)} \right)=k, H(\bm{0})\neq 0$ and $\mathcal{B} \in con^* \left(\{ l^{(1)},\dots,l^{(q+w)} \} \right)$, then $\operatorname{deg}(Q_{\bm{\beta}}) =  q + w - k$.
\end{enumerate}
\end{lemma}
With these results in hand, we obtain precise estimates for the weighted solutions to the equation $x_1\cdots x_k= \square$.
\begin{lemma}\label{count_squares}
    Let $k\geq 2$ and $\bm{\beta}=(\beta_1,\dots,\beta_k)$ a $k$-tuple of strictly positive real numbers. Let $\alpha:= \sum_{i=1}^{k}\frac{\beta_i}{2}.$ Then there exist a polynomial $Q_{\bold{\beta}}$ of degree $\frac{k(k-1)}{2}$ and  $\delta_k>0$ such that for large $Y$, we have 
    $$\sum_{n_1n_2\cdots n_k =\square \atop 1\leq n_i \leq Y^{\beta_i}} \prod_{p \mid n_1 \cdots n_k}\left(1+1/p\right)^{-1} = Y^{\alpha}Q_{\bm{\beta}}(\log Y) + O(Y^{\alpha-\delta_k}).$$
\end{lemma}

In order to prove Lemma \ref{count_squares}, we review some basic facts of arithmetical functions in several variables. For a positive integer $r$, let $f : \mathbb{N}^r \to \mathbb{C}$ be an arithmetic function of $r$ variables. Then, $f$, which is not identically zero, is said to be multiplicative if $f(1,\dots,1)=1$ and
\[
f(m_1n_1,\dots,m_rn_r) = f(m_1,\dots,m_r)f(n_1,\dots,n_r)
\]
holds for any $m_1,\dots,m_r, n_1,\dots,n_r \in \mathbb{N}$ such that $\gcd(m_1 \cdots m_r, n_1 \cdots n_r) = 1$. This is a generalization of the classical one-variable multiplicative functions that satisfy $f(1)=1$ and $f(mn) = f(m)f(n)$ for $\gcd(m,n)=1$. 

If $f$ is multiplicative, then it is determined by the values $f(p^{v_1}, \dots, p^{v_r})$, where $p$ is prime and $v_1, \dots,v_r \in \mathbb{N}_0$. Thus, a formal Dirichlet series in several variables of $f$ admits an Euler product expansion:
\begin{align}
\label{Euler prod}
\sum_{n_1, \dots n_r =1}^\infty \frac{f(n_1,\dots, n_r)}{n_1^{s_1}\cdots n_r^{s_r}} =\prod_{p} \left( \sum_{v_1,\dots,v_r = 0}^\infty \frac{f(p^{v_1}, \dots, p^{v_r})}{p^{v_1s_1+\dots+v_rs_r}}\right).
\end{align}

In the literature on multiplicative functions in several variables, the corresponding multiple Dirichlet series appears for the first time in the paper by Vaidyanathaswamy~\cite{V31}. The theory of multiple Dirichlet series was further developed by some authors without mentioning~\cite{V31}. In 1977, Selberg formulated multiple Dirichlet series attached to multiplicative functions in several variables (see~\cite[Chapter I, Definition 4.17]{T}). For details of the theory of multiple Dirichlet series, we refer to~\cite{To}.

\begin{proof}[Proof of Lemma \ref{count_squares}]
Let $f(n_1,\dots,n_k) := \prod_{p \mid n}\left(1+1/p \right)^{-1} \mathds{1}_{n=\square}$. We consider a multiple Dirichlet series associated to $f$ as
\begin{align*}
F(s_1,\dots,s_k) &:= \sum_{n_1, \dots n_{k} =1}^\infty \frac{f(n_1,\dots, n_{k})}{n_1^{s_1}\cdots n_{k}^{s_{k}}}. 
\end{align*}
Since $f(n_1,\dots,n_k)$ is multiplicative, $F(s_1,\dots,s_k)$ has the following Euler product expansion:
\begin{align}
\label{multiple-F}
\begin{split}
F(s_1,\dots,s_k) &= \prod_{p} \left( 1+ \sum_{n=1}^\infty \sum_{\substack{v_1,\dots,v_{k} =0 \\ v_1+\dots+v_{k} = 2n}}^\infty \left(\frac{p}{p+1}\right) \frac{1}{p^{v_1s_1+\dots+v_{k}s_{k}}} \right) \\
    &= \left(\prod_{1 \leq j \leq k} \zeta(2s_j) \prod_{1 \leq l_1<l_2 \leq k} \zeta(s_{l_1}+s_{l_2})\right) E(s_1,\dots,s_{k}), 
\end{split}
\end{align}
where 
\begin{align}
\label{E(s)}
\begin{split}
E(s_1,\dots,s_{k}) 
&= \prod_{p} \left( 1+ \sum_{n=1}^\infty \sum_{\substack{v_1,\dots,v_{k} =0 \\ v_1+\dots+v_{k} = 2n}}^\infty \left(\frac{p}{p+1}\right) \frac{1}{p^{v_1s_1+\dots+v_{k}s_{k}}} \right) \\
& \quad \times \prod_{1 \leq j \leq k} \left( 1-\frac{1}{p^{2s_j}}\right) \prod_{1 \leq l_1<l_2 \leq k} \left( 1-\frac{1}{p^{s_{l_1}+s_{l_2}}}\right).
\end{split}
\end{align}
By putting $\sigma = \min \{ \sigma_j \mid 1 \leq j \leq k\}$ we have
\begin{align*}
\sum_{n=2}^\infty \sum_{\substack{v_1,\dots,v_{k} =0 \\ v_1+\dots+v_{k}=2n}}^\infty \left(\frac{p}{p+1}\right) \frac{1}{p^{v_1s_1+\dots+v_{k}s_{k}}} 
&\ll \sum_{n=2}^\infty \frac{1}{p^{2n\sigma}} \sum_{\substack{v_1,\dots,v_{k}\geq 0 \\ v_1+\dots+v_{k}=2n}}1 \\
&\ll_k \sum_{n=2}^\infty \frac{(2n+1)^{k}}{p^{2n\sigma}} \ll_k \frac{1}{p^{4 \sigma}}.
\end{align*}
Hence, we have
\begin{align*}    
&E(s_1,\dots,s_{k}) \\
&= \prod_{p} \left( 1+\sum_{\substack{v_1,\dots,v_{k} \geq 0 \\ v_1+\dots+v_{k}=2}} \left(\frac{p}{p+1}\right) \frac{1}{p^{v_1s_1+\dots+v_{k}s_{k}}} + \sum_{n=2}^\infty \sum_{\substack{v_1,\dots,v_{k} =0 \\ v_1+\dots+v_{k} = 2n}}^\infty \left(\frac{p}{p+1}\right) \frac{1}{p^{v_1s_1+\dots+v_{k}s_{k}}} \right) \\
&\qquad \times \prod_{1 \leq j \leq k} \left( 1-\frac{1}{p^{2s_j}}\right) \prod_{1 \leq l_1<l_2 \leq k} \left( 1-\frac{1}{p^{s_{l_1}+s_{l_2}}}\right) \\
&=\prod_{p} \left( 1+\sum_{\substack{1 \leq i \leq k \\ \\ \\}} \left(1-\frac{1}{p+1}\right) \frac{1}{p^{2s_i}} +\sum_{1 \leq h_1<h_2 \leq k} \left(1-\frac{1}{p+1}\right) \frac{1}{p^{s_{h_1}+s_{h_2}}} +O_k\left(\frac{1}{p^{4\sigma}}\right) \right) \\
&\qquad \times \left(1 - \sum_{1 \leq j \leq k} \frac{1}{p^{2s_j}}+ O_k\left(\frac{1}{p^{4\sigma}}\right)\right) \left(1-\sum_{1 \leq l_1<l_2 \leq k} \frac{1}{p^{s_{l_1}+s_{l_2}}}+O_k\left( \frac{1}{p^{4\sigma}}\right)\right) \\
&= \prod_{p} \left(1+O_k\left( \frac{1}{p^{4\sigma}}\right)\right).
\end{align*}
Therefore, $E(s_1,\dots,s_{k})$ is absolutely convergent for $\re(s_j)>1/4$.

From (\ref{multiple-F}), we find that the series $F(s)$ converges absolutely for $\re(s_j) > 1/2$ for all $1 \leq j \leq k$ and thus satisfies (1) of Lemma \ref{thm:breteche1}. 

Next, we write $\bm{1/2} = (1/2,\dots,1/2)$. Then $F(\bm{s} + \bm{1/2})$ is an absolutely convergent series for $\re(s_j)> 0$. Therefore, we define the function
\begin{align}
\label{H(s)}
H(\bm{s}):=F(\bm{s} + \bm{1/2})\left(\prod_{1 \leq j \leq k}2s_j \right)\left( \prod_{1 \leq l_1 <l_2 \leq k}(s_{l_1}+s_{l_2}) \right).    
\end{align}
So, we take $\mathcal{L} = \left(l^{(i)}(\bm{s})\right)_{1 \leq i \leq r+k} := \{2s_j \mid 1 \leq j \leq k\} \cup \{s_{l_1}+s_{l_2} \mid 1 \leq l_1<l_2 \leq k \}$ in (2) of Lemma \ref{thm:breteche1}.
Then, by (\ref{multiple-F}) it can be rewritten as
\begin{align*}
    H(\bm{s})&=\left(\prod_{1 \leq j \leq k}\zeta(2s_j+1)2s_j \right)\left( \prod_{1 \leq l_1 <l_2 \leq k}\zeta(s_{l_1}+s_{l_2}+1)(s_{l_1}+s_{l_2}) \right)E(\bm{s}+\bm{1/2}).
\end{align*}
For $j \in \{1,\dots,k\}$, there exists $\delta_1 \in (0,1/4)$ such that $\zeta(2s_j+1)2s_j$ has analytic continuation to the plane $\re(s_j)>-\delta_1$. Similarly, $\zeta(s_{l_1}+s_{l_2}+1)(s_{l_1}+s_{l_2})$ also has analytic continuation to the plane $\re(s_j)>-\delta_1$. 

Furthermore, $E(\bm{s}+\bm{1/2})$ is holomorphic in $\re(s_j)>-\delta_1$ for $1 \leq j \leq k$. Since $c_j = 1/2$ for all $1 \leq j \leq k$, we can take $h^{(j)}(\bm{s}) = s_j$ in the notation of Lemma \ref{thm:breteche1}. Then we put $\delta_3=\delta_1$. Therefore, $F$ also satisfies (2) of Lemma \ref{thm:breteche1}.

Finally, for $\re(s_j) > -1/4$, by applying the convexity bound of the Riemann zeta-function,
\begin{align*}
\zeta(s_j+1)2s_j &\ll (\abs{s_j}+1)^{1- \frac{1}{2}\min \{0, \re(s_j)\}+\varepsilon}, \\
\zeta(s_{l_1}+s_{l_2}+1)(s_{l_1}+s_{l_2}) &\ll (\abs{s_{l_1}+s_{l_2}}+1)^{1- \frac{1}{2}\min \{0, \re(s_{l_1}+s_{l_2})\}+\varepsilon}
\end{align*}
hold. The above argument shows that $H(s)$ satisfies (3) of Lemma \ref{thm:breteche1} with $\delta_2 = 1/2$.

Since $q= k+\binom{k}{2} = k+k(k-1)/2, w=0$ and the rank of the linear forms in $\mathcal{L}$ is $k$, we have 
\[
\sum_{1 \leq n_1 \leq Y^{\beta_1}}\dots \sum_{1 \leq n_k \leq Y^{\beta_k}} f(n_1,\dots,n_k) = Y^{\alpha} Q_{\bm{\beta}}(\log Y) + O(Y^{\alpha-\delta_k}),
\]
where $Q_{\bm{\beta}}(\log X)$ is a polynomial of degree at most $k(k-1)/2$.

The remaining task is to determine the degree of the polynomial $Q$ by using Lemma \ref{thm:breteche2}. Since $\bm{c}=\bm{1/2}$, we know that $w=0$, and then $Rank\left(l^{(1)},\dots,l^{(q)} \right) =k$. Moreover as $s_j, s_{l_1}+s_{l_2} \to 0$, we have
\begin{align*}
    \zeta(2s_j+1)2s_j \to 1, \quad \zeta(s_{l_1}+s_{l_2}+1)(s_{l_1}+s_{l_2}) \to 1.
\end{align*}
From the Euler product, it holds that $E(\bm{1/2})$ does not vanish. Hence $H(\bm{0})\neq 0$. At last, we see that $\mathcal{B} = \sum_{i=1}^{k} \bm{e}_i^*(\bm{s}) \in con^* \left(\{ l^{(1)},\dots,l^{(q)} \} \right)$ for $\bm{e}_i^*(\bm{s}) = s_i$. Therefore, by Lemma \ref{thm:breteche2} (2), we have $\operatorname{deg}(Q) = k(k-1)/2$ which completes the proof.

\end{proof}

We record the main result of \cite[Corollary $1.2$]{GZ} about bounds on shifted moments of $L$- functions that we state in a special case which is sufficient for our application.
\begin{proposition}\label{th_shifted}
 Under the truth of GRH, let $k\geq 1$ be a fixed integer and $A$ be a fixed positive real number. Suppose that $X$ is a large real number and $t=(t_1,\ldots ,t_{k})$ a real $k$-tuple with $|t_j|\leq  X^A$. Then
\begin{align*}
\begin{split}
  \sideset{}{^*}{\sum}_{\substack{(d,2)=1 \\ d \leq X}} & \big| L\big( \tfrac{1}{2} +\rm{i}t_1,\chi^{(8d)} \big) \big| \cdots \big| L\big(\tfrac{1}{2}+\rm{i}t_{k},\chi^{(8d)}  \big) \big| \\
& \ll  X(\log X)^{k/4} \prod_{1\leq i<j\leq k} g(|t_i-t_j|)^{1/2}g(|t_i+t_j|)^{1/2}\prod_{1\leq i\leq k} g(|2t_i|)^{3/4},
\end{split}
\end{align*}
where $g:\mathbb{R}_{\geq 0} \rightarrow \mathbb{R}$ is defined by
\begin{equation} \label{gDef}
g(x) =\begin{cases}
\log X,  & \text{if } x\leq 1/\log X \text{ or } x \geq e^X, \\
1/x, & \text{if }   1/\log X \leq x\leq 10, \\
\log \log x, & \text{if }  10 \leq x \leq e^{X}.
\end{cases}
\end{equation}
Here, the implied constant depends on $k$ and $A$, but not on $X$ or the $t_j$'s.
\end{proposition}

\section{Proofs of the lower bounds} 
For any $Y\geq 1$ and any fundamental discriminant $d>0$, let \begin{equation}\label{def_carac}S_{\chi_d}(Y):= \sum_{n \leq Y} \chi_d(n).\end{equation}
 \subsection{Proof of Theorem \ref{lowquad}}
 We define for a fundamental discriminant $d$ the character sum  $$
M_{\varepsilon,Y}(\chi_d)= S_{\chi_d}(Y^{\varepsilon})
$$
and consider the following sums
\begin{equation}
\label{eq:S1S2}
\mathcal{S}_1= \sideset{}{^*}{\sum}_{0<d\leq X} S_{\chi_d}(Y) M_{\varepsilon,Y}(\chi_d)^{k-1} \mand \mathcal{S}_2 =\sideset{}{^*}{\sum}_{0<d\leq X} \vert M_{\varepsilon,Y}(\chi_d) \vert^k.
\end{equation} By H\"older inequality, we get
$$ \mathcal{S}_{1}^{k} \leq \mathcal{S}_2 ^{k-1} \sideset{}{^*}{\sum}_{0<d\leq X} \vert S_{\chi_d}(Y)\vert^k.
$$  Theorem \ref{lowquad} follows from Lemma~\ref{lem:1-2moments weight}  and Lemma~\ref{lem:1-2moments} which yield the lower bound $$\sideset{}{^*}{\sum}_{0<d\leq X} \vert S_{\chi_d}(Y) \vert^{k}  \gg XY^{k/2}(\log X)^{\frac{k(k-1)}{2}}.
$$

\begin{lemma}\label{lem:1-2moments weight} 
For any even integer $k \ge 2$ and a  sufficiently small $\varepsilon>0$, we have 
$$
\mathcal{S}_2 \ll X^{1+\varepsilon k/2}(\log X)^{\frac{k(k-1)}{2}}
$$ where the implied constant depends only on $\varepsilon$ and $k$.
 \end{lemma}
 
 \begin{proof}  
  We have \begin{align*}  
 \mathcal{S}_2  & =  \sideset{}{^*}{\sum}_{0<d\leq X} \,\,\sum_{x_1,\dots,x_k \leq Y^{\varepsilon}} \chi_d(x_1\cdots x_k)
   = \mathcal{S}_2 ^{sq} + \mathcal{S}_2^{nsq} \end{align*} where 
\begin{equation}
\label{eq:S2sq}
\mathcal{S}_2 ^{sq}= \sideset{}{^*}{\sum}_{0<d\leqslant X} \,\,\sum_{x_1,\dots,x_k \leq Y^{\varepsilon}\atop x_1\dots x_k=\square} \chi_d\left(\prod_{i=1}^{k}x_i\right)
\end{equation}  and
\begin{equation}
\label{eq:S2nsq}
\mathcal{S}_2^{nsq}=\sideset{}{^*}{\sum}_{0<d\leqslant X} \,\, \sum_{x_1,\dots,x_k \leq Y^{\varepsilon}\atop x_1\cdots x_k\neq\square} \chi_d\left(\prod_{i=1}^{k}x_i\right).
 \end{equation}  Switching the summation, using Lemma \ref{sqrfree} and the fact that $Y \leq X$, we have for any $\delta>0$  $$\mathcal{S}_2^{nsq} \ll X^{1/2} \sum_{x_1,\dots,x_k \leq Y^{\varepsilon}} (x_1\cdots x_k)^{1/4+\delta} \ll X^{1/2 +\frac{5k \varepsilon}{4}+k \varepsilon\delta}= X^{1/2+o(1)}.$$ Moreover, by Lemma \ref{count_squares} with $\beta_j = \varepsilon$ for all $1 \leq j \leq k$ we get $$\mathcal{S}_2^{sq} \ll X\sum_{x_1,\dots,x_k \leq Y^{\varepsilon}\atop x_1\cdots x_k=\square} \prod_{p\mid x_1\cdots x_k} (1+1/p)^{-1} \ll XY^{\varepsilon k/2}(\log X)^{\frac{k(k-1)}{2}}. $$ 
 \end{proof}

 \begin{lemma}\label{lem:1-2moments} 
 Let $X,Y$ such that $X^{\varepsilon} \ll Y \ll X^{2/3-\alpha}$ for some $\alpha>0$.
For any integer $k \ge 1$ and a  sufficiently small $\varepsilon>0$, we have 
$$
\mathcal{S}_{1}\gg XY^{1/2+\varepsilon \frac{(k-1)}{2}}(\log X)^{\frac{k(k-1)}{2}}
$$ where the implied constant depends only on $\varepsilon$ and $k$. Assuming GRH, the result remains true under the weaker hypothesis $X^{\varepsilon} \ll Y \ll X^{1-\alpha}$ for some $\alpha>0$.
 \end{lemma}

 \begin{proof}  
Expanding the summations and proceeding as in the proof of Lemma \ref{lem:1-2moments weight}, we split the summation into two sums depending on whether the product of the variables is a square or not. We obtain $$\mathcal{S}_1=\mathcal{S}_1^{sq} + \mathcal{S}_1^{nsq}$$ where we have 
\begin{align*}  
\mathcal{S}_1^{sq}:&=\sideset{}{^*}{\sum}_{0<d\leq X} \ssum_{\substack{n\leq  Y\\ x_1,\ldots,x_{k-1}\leq Y^{\varepsilon} \\ nx_1\cdots x_{k-1}=\square}} \chi_d(nx_1\cdots x_{k-1}) \end{align*} and \begin{align*}
\mathcal{S}_1^{nsq}:=\sideset{}{^*}{\sum}_{0<d\leq X} \ssum_{\substack{n\leq  Y \\ x_1,\ldots,x_{k-1}\leq Y^{\varepsilon} \\ nx_1\cdots x_{k-1} \neq \square}} \chi_d(nx_1\cdots x_{k-1}). \end{align*} On one hand by Lemma \ref{sqrfree} we obtain
\begin{align*}  
\mathcal{S}_1^{sq} & \gg X\ssum_{\substack{n\leq  Y \\ x_1,\ldots,x_{k-1}\leq Y^{\varepsilon} \\ nx_1\cdots x_{k-1}=\square}} \prod_{p\mid nx_1\cdots x_{k-1}} (1+1/p)^{-1} \gg XY^{1/2+\varepsilon \frac{(k-1)}{2}}(\log Y)^{\frac{k(k-1)}{2}}
\end{align*} where we again applied Lemma \ref{count_squares}. On the other hand, if $nx_1\dots x_{k-1} \neq \square$,  Lemma \ref{sqrfree} implies that for any $\delta>0$, the following upper bound holds
$$ \sideset{}{^*}{\sum}_{0<d\leq X} \chi_d(nx_1\cdots x_{k-1}) \ll X^{1/2}  (nx_1\cdots x_{k-1})^{1/4+\delta}.$$ Thus we get,
\begin{align*} \mathcal{S}_1^{nsq}& \ll X^{1/2}\ssum_{\substack{n\leq  Y \\ x_1,\ldots,x_{k-1}\leq Y^{\varepsilon}}}(nx_1\cdots x_{k-1})^{1/4+\delta} \\
&  \ll X^{1/2}Y^{5/4+\frac{5(k-1)\varepsilon}{4}+(k-1)\epsilon \delta +\delta } \ll Y^{1/2}X^{1-\alpha} \end{align*} for some $\alpha>0$. This concludes the unconditional proof. \\
Under GRH, we proceed in the same way until the last step where we appeal to the second part of Lemma \ref{sqrfree} to get 
$$ \mathcal{S}_1^{nsq} \ll X^{1/2+\varepsilon} \ssum_{\substack{n\leq  Y \\ x_1,\ldots,x_{k-1}\leq Y^{\varepsilon}}}(nx_1\cdots x_{k-1})^{\delta} \ll X^{1/2+\varepsilon} Y^{1+\delta+(k-1)\varepsilon(1+\delta)}$$ which is negligible compared to $\mathcal{S}_1^{sq} $ if $Y \ll X^{1-\alpha}$ for some $\alpha>0$.
\end{proof}

\subsection{Proof of Theorem \ref{lowtheta}}
In order to lower bound $\mathcal{S}_1$, we use the following approximation of 
$\vartheta(1,\chi_d)$ by a truncated sum, which easily
follows from estimating the tail via the corresponding 
integral. 

\begin{lemma}
\label{lem:approx}
Let $\delta>0$ be a positive number.
 Then 
$$
\vartheta(1,\chi_d)=\sum_{n\leq d^{1/2+\delta}} 
\chi_d(n)e^{-\pi n^2/d} + O(d^{1/2}e^{-d^{\delta}}).
$$
\end{lemma}
We proceed as in the proof of Theorem \ref{lowquad} and let 
$ M_{\varepsilon,X}(\chi_d)= S_{\chi_d}(X^{\epsilon}).$ Consider the following sums\begin{equation}
\label{eq:S1S2theta}
\mathcal{S}_1= \sideset{}{^*}{\sum}_{X/2<d\leq X} \vartheta(1,\chi_d) M_{\varepsilon,X}(\chi_d)^{k-1} \mand \mathcal{S}_2 =\sideset{}{^*}{\sum}_{X/2<d\leq X} \vert M_{\varepsilon,X}(\chi_d) \vert^k.
\end{equation} By H\"older inequality, we get
$$ \mathcal{S}_{1}^{k} \leq \mathcal{S}_2 ^{k-1} \sideset{}{^*}{\sum}_{X/2<d\leq X} \vert \vartheta(1,\chi_d)\vert^k.
$$ Theorem \ref{lowtheta} follows from Lemma~\ref{lem:1-2moments weight}  and Lemma~\ref{lem:1-2moments-theta} below.

 \begin{lemma}\label{lem:1-2moments-theta} 
For any integer $k \ge 2$ and a  sufficiently small $\varepsilon>0$, we have 
$$
\mathcal{S}_{1}\gg X^{5/4+\varepsilon\frac{(k-1)}{2}}(\log X)^{\frac{k(k-1)}{2}}.
$$ where the implied constant depends only on $\varepsilon$ and $k$.
 \end{lemma}

 \begin{proof}  
By Lemma \ref{lem:approx} with $\delta=1/12$, we have \begin{equation}\mathcal{S}_{1} = \sideset{}{^*}{\sum}_{X/2<d\leq X} \left(\sum_{n\leq d^{7/12}}\chi_d(n)e^{-\pi n^2/d}\right) M_{\varepsilon,X}(\chi_d)^{k-1} + R(X) \end{equation} with 
\begin{align*} R(X):= &\sideset{}{^*}{\sum}_{X/2<d\leq X} \vert M_{\varepsilon,X}(\chi_d)\vert^{k-1}d^{1/2}e^{-d^{1/12}} \ll X^{3/2+\varepsilon (k-1)}e^{-X^{1/12}}\ll 1.
\end{align*} As in the proof of Lemma \ref{lem:1-2moments weight}, we split the summation into two sums depending on whether the product of the variables is a square or not.
On one hand, we have 
\begin{align*}  
\mathcal{S}_1^{sq}:&=\sideset{}{^*}{\sum}_{X/2<d\leq X}   \ssum_{\substack{n\leq  d^{7/12} \\ x_1,\ldots,x_{k-1}\leq X^{\varepsilon} \\ nx_1\cdots x_{k-1}=\square}} e^{-\pi n^2/d}\chi_d(nx_1\cdots x_{k-1}). \end{align*} Restricting the sum to $n\leq X^{1/2}$ and using Lemma \ref{count_squares} leads to \begin{align*}  
\mathcal{S}_1^{sq} &\gg \ssum_{\substack{n\leq  X^{1/2} \\ x_1,\ldots,x_{k-1}\leq X^{\varepsilon} \\ nx_1\cdots x_{k-1}=\square}} \sideset{}{^*}{\sum}_{0<d\leq X} \chi_d(nx_1\cdots x_{k-1}) \\
& \gg X\ssum_{\substack{n\leq  X^{1/2} \\ x_1,\ldots,x_{k-1}\leq X^{\varepsilon} \\ nx_1\cdots x_{k-1}=\square}} \prod_{p\mid nx_1\cdots x_{k-1}} (1+1/p)^{-1} \gg X^{5/4+\varepsilon \frac{(k-1)}{2}}(\log X)^{\frac{k(k-1)}{2}}.
\end{align*} On the other hand
\begin{align*}
\mathcal{S}_1^{nsq}:=\sideset{}{^*}{\sum}_{X/2<d\leq X}   \ssum_{\substack{n\leq  d^{7/12} \\ x_1,\ldots,x_{k-1}\leq X^{\varepsilon} \\ nx_1\cdots x_{k-1} \neq \square}} e^{-\pi n^2/d}\chi_d(nx_1\cdots x_{k-1}). \end{align*} By partial summation and Lemma \ref{sqrfree}, we see that for $nx_1\cdots x_{k-1} \neq \square$ and any $\delta>0$, the following upper bound holds
$$  \sideset{}{^*}{\sum}_{n^{12/7}<d\leq X}\chi_d(nx_1\cdots x_{k-1}) e^{-\pi n^2/d} \ll X^{1/2}  (nx_1\cdots x_{k-1})^{1/4+\delta}.$$ Thus,
$$ \mathcal{S}_1^{nsq} \ll X^{1/2}\ssum_{\substack{n\leq  X^{7/12} \\ x_1,\ldots,x_{k-1}\leq X^{\varepsilon}}}(nx_1\cdots x_{k-1})^{1/4+\delta} \ll X^{59/48+\frac{5(k-1)\varepsilon}{4}+(k-1)\epsilon \delta+\delta } \ll X^{5/4-\alpha}$$ for some $\alpha>0$. This concludes the proof.
\end{proof}

\section{Proof of the asymptotic formula}
We determine the leading coefficient of the polynomial $Q_{\bm{\beta}}$ of Lemma \ref{count_squares}. The leading coefficient of the average of quadratic twists of the M\"obius function was also given by the second named author in~\cite{Toma25}. For a wide class of multivariable arithmetic functions, Essouabri, Salinas Zavala and T\'oth~\cite{EST22} gave the leading coefficient of asymptotic behavior of their multiple averages.

Let $r=k(k-1)/2$ and let $\mathcal{H} = \left(h^{(i)}\right)_{1 \leq i \leq r}$ be a subfamily of $\mathcal{L}_0 = \left(l^{(i)}\right)_{1 \leq i \leq r+k}$ defined by $\mathcal{L}_0 \setminus \{ 2\bm{e}_1, \dots, 2\bm{e}_k \}$. 

\begin{lemma}\label{leading_coefficient}
Let $k\geq 2$. The leading coefficient of the polynomial $Q_{\bold{\beta}}$ in Lemma \ref{count_squares} is given by
\begin{equation*}
    \prod_{p} \frac{\left(1- \frac{1}{p} \right)^{\frac{k(k+1)}{2}}}{1+\frac{1}{p}} \left(\frac{1}{p}+\frac{1}{2}\left[\left(1-\frac{1}{\sqrt{p}}\right)^{-k}+\left(1+\frac{1}{\sqrt{p}}\right)^{-k} \right]\right) \mathcal{I}_{\bm{\beta}}
\end{equation*}
where 
\begin{align}
\label{I_beta}
    \mathcal{I}_{\bm{\beta}} &= \int \cdots\int_{\mathcal{A}_{\bm{\beta}}} du_1\cdots du_{r},
\end{align}
and
\[
\mathcal{A}_{\bm{\beta}} = \left\{ (u_1,\dots,u_{r}) \in [0,\infty)^{r} \middle| \sum_{i=1}^{r} h^{(i)}(\bm{e}_j) u_i \leq \beta_j \text{ for all } 1 \leq j \leq r \right\}.
\]
\end{lemma}

\begin{proof}
We denote the leading coefficient of $Q_{\bm{\beta}}$ by $C(k)$. We can easily check that assumptions (C1) and (C2) in Lemma \ref{thm:breteche2} (1) are satisfied. In fact, by the definition of (\ref{H(s)}), it is clear that such a function $G$ exists. Also, for $\mathcal{L}_0 =\left( l^{(i)}\right)_{1\leq i \leq r+k}$, we already showed that $Rank(\mathcal{L}_0)=k$ and $\#\mathcal{L}_0= \binom{k+1}{2}$. We need to show that there is no subfamily $\mathcal{L}^\prime$ of $\mathcal{L}_0$ such that $\mathcal{L}^\prime \neq \mathcal{L}_0$, $\mathcal{B} \in Vect(\mathcal{L}^\prime)$ and $\#\mathcal{L}^\prime-Rank(\mathcal{L}^\prime)=\frac{k(k-1)}{2}$. If such a family exists, then we must have $n=\#\mathcal{L}^\prime \geq \frac{k(k-1)}{2}$. Let us first assume that all forms $e_i^*+e_j^*, 1\leq i<j \leq k$ lie in $\mathcal{L}^\prime$. This would imply $Rank(\mathcal{L}^\prime)=k$ and $\#\mathcal{L}^\prime-Rank(\mathcal{L}^\prime)<\frac{k(k+1)}{2}-k=\frac{k(k-1)}{2}$. Hence, we can assume that $\mathcal{L}^\prime$ consists of $\ell>n-\frac{k(k-1)}{2}$ linear forms of type $e_i^*$ and $n-\ell<\frac{k(k-1)}{2}$ forms of type $e_i^*+e_j^*$.
Thus, we have $n-Rank(L') \leq n-\ell < \frac{k(k-1)}{2}$ and (C2) is verified.

Hence, applying Lemma \ref{thm:breteche2} (1) and using that $l^{(i)}(\bm{c})=1$ for all $1 \leq i \leq r+k$, we have 
\begin{align*}
    \mathcal{I}_{\bm{\beta}} &= \lim_{x \to \infty} \frac{1}{x^{\frac{\sum_{i=1}^{k}\beta_i}{2}} (\log x)^{\frac{k(k-1)}{2}}}\int_{\mathcal{A}_{\bm{\beta}}(x)} dy_1\cdots dy_{r+k},
\end{align*}
where
\[
\mathcal{A}_{\bm{\beta}}(x) = \left\{ (y_1,\dots,y_{r+k}) \in [1,\infty)^{r+k} \middle| \prod_{i=1}^{r+k} y_{i}^{l^{(i)}(\bm{e}_j)} \leq x^{\beta_j} \text{ for all } 1 \leq j \leq k \right\}.
\]
Here, $l^{(i)}(\bm{e}_j)=2$ for a linear form, $l^{(i)}(\bm{e}_j)=1$ for $k-1$ linear forms in $\mathcal{L}$, and $l^{(i)}(\bm{e}_j)=0$ otherwise. So, after reindexing and rearranging $y_j$'s, we have
\begin{align*}
     \prod_{i=1}^{r+k} y_{i}^{l^{(i)}(\bm{e}_j)} &=y_j^2 y_{k+1}^{a_{k+1}(j)} \cdots y_{r+k}^{a_{r+k}(j)}
\end{align*}
for $1 \leq j \leq k$, where $a_{k+1}(j), \dots, a_{r+k}(j) \in \{0,1 \}$ and $\# \{ i \mid a_i(j)=1\} =k-1$. Hence, letting $\bm{y}=(y_1,\dots,y_{r+k})$, we have
\begin{align*}
    \int_{\mathcal{A}_{\bm{\beta}}(x)} d\bm{y} &= \int \cdots \int_{\substack{y_{k+1}, \dots, y_{r+k} \in [1,\infty)\\ \prod_{i=k+1}^{r+k} y_{i}^{a_{i}(j)}  \leq x^{\beta_j} (\forall j)}} \left( \prod_{j=1}^k \int_{1 \leq y_j \leq \sqrt{\frac{x^{\beta_j}}{y_{k+1}^{a_{k+1}(j)} \cdots y_{r+k}^{a_{r+k}(j)}}}} dy_j \right) dy_{r+k} \cdots dy_{k+1} \\
    &=\int \cdots\int_{\substack{y_{k+1}, \dots, y_{r+k} \in [1,\infty)\\ \prod_{i=k+1}^{r+k} y_{i}^{a_{i}(j)}  \leq x^{\beta_j} (\forall j)}} \prod_{j=1}^k \left(  \sqrt{\frac{x^{\beta_j}}{y_{k+1}^{a_{k+1}(j)} \cdots y_{r+k}^{a_{r+k}(j)}}} -1 \right) dy_{r+k} \cdots dy_{k+1}.
\end{align*} Noting that $\# \{ j \mid a_i(j)=1\} =2$ for all $k+1 \leq i \leq r+k$, the integrand in the above can be written as
\begin{align*}
\prod_{j=1}^k \left(  \sqrt{\frac{x^{\beta_j}}{y_{k+1}^{a_{k+1}(j)} \cdots y_{r+k}^{a_{r+k}(j)}}} -1 \right) & = \frac{x^{\sum_{j=1}^k \frac{\beta_j}{2}}}{y_{k+1}\cdots y_{r+k}}\left( 1 + \sum_{d=1}^k (-1)^d R(j_1,\dots,j_d)\right).
\end{align*}
where \begin{equation}\label{def-ji}
    R(j_1,\dots,j_d):= \frac{\sqrt{\prod_{i=1}^r y_{k+i}^{h^{(i)}(\bm{e}_{j_1})+\dots+h^{(i)}(\bm{e}_{j_d})}}}{x^{\frac{\beta_{j_1}}{2}+\dots+\frac{\beta_{j_d}}{2}}}.
\end{equation}
Integrating the first term into the brackets gives us the main term 
\begin{align*}
    & \int \cdots\int_{\substack{y_{k+1}, \dots, y_{r+k} \in [1,\infty)\\ y_{k+1}^{a_{k+1}(j)} \cdots y_{r+k}^{a_{r+k}(j)} \leq x^{\beta_j} (\forall j)}} \prod_{j=1}^k \left(  \sqrt{\frac{x^{\beta_j}}{y_{k+1}^{a_{1}(j)} \cdots y_{r+k}^{a_{r+k}(j)}}} \right) dy_{r+k} \cdots dy_{k+1} \\
    &= x^{\sum_{1 \leq j \leq k}\frac{\beta_j}{2}}\int \cdots\int_{\substack{y_{k+1}, \dots, y_{r+k} \in [1,\infty)\\ y_{k+1}^{a_{k+1}(j)} \cdots y_{r+k}^{a_{r+k}(j)} \leq x^{\beta_j} (\forall j)}} \frac{1}{y_{k+1} \cdots y_{r+k}} dy_{r+k} \cdots dy_{k+1}  \end{align*}
We now make the substitution $y_{k+i} = \exp(u_i \log x) $ for all $1 \leq i \leq r$ to obtain
    \begin{align*}
x^{\sum_{1 \leq j \leq k} \frac{\beta_j}{2}} (\log x)^{r} \int \cdots \int_{\substack{u_1, \dots, u_{r} \in [0,\infty) \\ \sum_{i=1}^{r}  h^{(i)}(\bm{e}_j) u_i\leq \beta_j \,\,(\forall j)}}  du_{1} \cdots du_{r}.
\end{align*}
Therefore, we get 
\begin{align*}
    \mathcal{I}_{\bm{\beta}} &= \int_{\mathcal{A}_{\bm{\beta}}} du_{1} \cdots du_{r}
\end{align*}
as claimed. 

It remains to show that for any $1 \leq d \leq k$ and $1\leq j_1,\dots,j_d \leq k$ the integral of 
\begin{align}\label{small contribution}
\frac{x^{\sum_{j=1}^k \frac{\beta_j}{2}}}{y_{k+1}\cdots y_{r+k}} R(j_1,\dots,j_d) 
\end{align}
gives a negligible contribution. Without loss of generality, we can assume that $j_1=1,\dots,j_d=d$. Then, setting $y_{k+i}=\exp \left( u_i \log x\right)$, we have
\begin{align*}
& \int \cdots\int_{\substack{y_{k+1}, \dots, y_{r+k} \in [1,\infty)\\ y_{k+1}^{a_{k+1}(j)} \cdots y_{r+k}^{a_{r+k}(j)} \leq x^{\beta_j} (\forall j)}} \frac{x^{\sum_{j=1}^k \frac{\beta_j}{2}}}{y_{k+1}\cdots y_{r+k}}\frac{\sqrt{\prod_{i=1}^r y_{k+i}^{h^{(i)}(\bm{e}_{1})+\dots+h^{(i)}(\bm{e}_{d})}}}{x^{\frac{\beta_{1}}{2}+\dots+\frac{\beta_{d}}{2}}} dy_{r+k} \cdots dy_{k+1} \\
&= x^{\sum_{j=d+1}^k \frac{\beta_j}{2}}(\log x)^r \int_{\substack{u_1, \dots,u_r \geq 0 \\ \sum_{i=1}^r h^{(i)}(\bm{e}_j)u_i \leq \beta_j \ (\forall j)}} \exp\left(\frac{\log x}{2}\sum_{i=1}^r \sum_{j=1}^d h^{(i)}(\bm{e}_j)u_i \right) du_1\cdots du_r.
\end{align*}
The conditions $\sum_{i=1}^r h^{(i)}(\bm{e}_j)u_i \leq \beta_j$ for $2 \leq j \leq d$ imply that the above integral is bounded by
\begin{align}\label{oneLinear}
& x^{\sum_{j=2}^k \frac{\beta_j}{2}}(\log x)^r \int_{\substack{u_1, \dots,u_r \geq 0 \\ \sum_{i=1}^r h^{(i)}(\bm{e}_j)u_i \leq \beta_j \ (\forall j)}} \exp\left(\frac{\log x}{2}\sum_{i=1}^r h^{(i)}(\bm{e}_1)u_i \right) du_1\cdots du_r.
\end{align}
From the definition of $h^{(i)}$, there exist $i_1 < \cdots<i_{k-1}$ such that $h^{(i_1)}(\bm{e}_1), \dots, h^{(i_{k-1})}(\bm{e}_1)=1$, and there also exists $j^\prime>1$ such that $h^{(i_{k-1})}(\bm{e}_{j^\prime})=1$ and $h^{(i_1)}(\bm{e}_{j^\prime}), \dots, h^{(i_{k-1})}(\bm{e}_{j^\prime})=0$. 
Let $$\mathcal{C}_{j'} =  \left\{\mathbf{u}=(u_1,\dots,u_r), i\neq i_{k-1}, \beta_1-\sum_{i \neq i_{k-1}}h^{(i)}(\bm{e}_1)u_i \leq \beta_{j^\prime}-\sum_{i \neq i_{k-1}}h^{(i)}(\bm{e}_{j^\prime})u_i \right\}.$$
We first integrate with respect to the variable $u_{i_{k-1}}$ in \eqref{oneLinear} to obtain
\begin{align*}
&
\int_{\substack{u_1, \dots,u_r \geq 0 \\ \sum_{i=1}^r h^{(i)}(\bm{e}_j)u_i \leq \beta_j \ (\forall j)}} \exp\left(  \frac{\log x}{2}\sum_{i=1}^r h^{(i)}(\bm{e}_1)u_i \right) du_1\cdots du_r \\
&= \frac{2x^{\frac{\beta_1}{2}}}{\log x} 
\int_{\substack{u_i\geq 0 \ (i \neq i_{k-1}), \ \mathbf{u} \in \mathcal{C}_{j'}\\ \sum_{i \neq i_{k-1}} h^{(i)}(\bm{e}_j)u_i \leq \beta_j \ (\forall j) \\ }} du_i \\
&\quad + \frac{2x^{\frac{\beta_{j^\prime}}{2}}}{\log x} \int_{\substack{u_i\geq 0 \ (i \neq i_{k-1}), \ \mathbf{u} \notin \mathcal{C}_{j'}\\ \sum_{i \neq i_{k-1}} h^{(i)}(\bm{e}_j)u_i \leq \beta_j \ (\forall j) }} e^{ \left(\frac{\log x}{2}\sum_{i \neq i_{k-1}} (h^{(i)}(\bm{e}_1)-h^{(i)}(\bm{e}_{j^\prime})) u_i \right)} du_i \\
&\quad-\frac{4}{\log x} \int_{\substack{u_i\geq 0 \ (i \neq i_{k-1}) \\ \sum_{i \neq i_{k-1}} h^{(i)}(\bm{e}_j)u_i \leq \beta_j \ (\forall j) }} e^{\frac{\log x}{2}\sum_{i \neq i_{k-1}} h^{(i)}(\bm{e}_1) u_i} du_i.
\end{align*}

The first integral is finite giving a contribution $O\left(x^{\frac{\beta_1}{2}}/\log x\right)$. The second integral is also evaluated as
\begin{align*}
&\frac{2x^{\frac{\beta_{j^\prime}}{2}}}{\log x} \int_{\substack{u_i\geq 0 \ (i \neq i_{k-1}), \ \mathbf{u} \notin \mathcal{C}_{j'} \\ \sum_{i \neq i_{k-1}} h^{(i)}(\bm{e}_j)u_i \leq \beta_j \ (\forall j)}} e^{ \left(\frac{\log x}{2}\sum_{i \neq i_{k-1}} (h^{(i)}(\bm{e}_1)-h^{(i)}(\bm{e}_{j^\prime})) u_i \right)} du_i \\
&\leq  \frac{2x^{\frac{\beta_1}{2}}}{\log x} 
\int_{\substack{u_i\geq 0 \ (i \neq i_{k-1}), \ \mathbf{u} \notin \mathcal{C}_{j'} \\ \sum_{i \neq i_{k-1}} h^{(i)}(\bm{e}_j)u_i \leq \beta_j \ (\forall j)}} du_i
\end{align*}
since $\beta_1-\sum_{i \neq i_{k-1}, i_{k-2}}h^{(i)}(\bm{e}_1)u_i \geq \beta_{j^\prime}-\sum_{i \neq i_{k-1}}h^{(i)}(\bm{e}_{j^\prime})u_i$. Hence, the contribution of the second integral is also $O\left(x^{\frac{\beta_1}{2}}/\log x\right)$. Repetition of this procedure leads to the result that 
\begin{align*}
    \frac{2}{\log x} \int_{\substack{u_i\geq 0 \ (i \neq i_{k-1}) \\ \sum_{i \neq i_{k-1}} h^{(i)}(\bm{e}_j)u_i \leq \beta_j \ (\forall j) }} e^{\frac{\log x}{2}\sum_{i \neq i_{k-1}} h^{(i)}(\bm{e}_1) u_i} du_i \ll \frac{x^{\beta_1}}{(\log x)^2}.
\end{align*}
Therefore, the integral (\ref{oneLinear}) is $O\left(x^{\sum \frac{\beta_j}{2}} (\log x)^{r-1} \right)$, and we find that the contribution coming from integrating the terms in (\ref{small contribution}) is negligible.

To conclude, since $\zeta(2s_j+1)2s_j \to 1$ as $s_j \to 0$ and $\zeta(s_{l_1}+s_{l_2})(s_{l_1}+s_{l_2}) \to 1$ as $s_{l_1}+s_{l_2} \to 0$, by (\ref{E(s)}) we have
\begin{align*}
H(\bm{0}) = E(\bm{1/2}) &=\prod_{p} \left( 1+ \left(1- \frac{1}{p+1}\right) \sum_{n=1}^\infty \binom{k+2n-1}{k-1}\frac{1}{p^n} \right) \left(1- \frac{1}{p} \right)^\frac{k(k+1)}{2} \end{align*}
Using that  $$\sum_{n=0}^{+\infty}  \binom{n+k}{k} x^n =\frac{1}{(1-x)^{k+1}}, \qquad \vert x\vert < 1, \qquad k\in \mathbb{N}$$ a quick computation reveals that 
 \begin{equation*} E(\bm{1/2})= \prod_{p} \frac{\left(1- \frac{1}{p} \right)^{\frac{k(k+1)}{2}}}{1+\frac{1}{p}} \left(\frac{1}{p}+\frac{1}{2}\left[\left(1-\frac{1}{\sqrt{p}}\right)^{-k}+\left(1+\frac{1}{\sqrt{p}}\right)^{-k} \right]\right).
\end{equation*}
Therefore, the proof is complete.
\end{proof}

\subsection{Proof of Theorem \ref{asymptotic for S_k(X,Y)}}
From the definition of $S_k(X,Y)$ we have
\begin{align*}
     S_k(X,Y) &= \sideset{}{^*}{\sum}_{0 <d \leq X} \left(\sum_{n \leq Y}\chi_d(n)\right)^{k} \\
     &= \sum_{n_1 \leq Y} \dots \sum_{n_{k} \leq Y} \,\,\sideset{}{^*}{\sum}_{0<d \leq X} \left(\frac{d}{n_1\cdots n_{k}}\right).
\end{align*}
Then we apply Lemma \ref{sqrfree} to obtain
\begin{align}
\label{expansion}
\begin{split}
S_k(X,Y) &= \frac{X}{\zeta(2)}\sum_{n_1,\dots,n_k \leq Y} \prod_{\substack{p \mid n_1\cdots n_{k} \\ n_1 \cdots n_{k}=\square}} \frac{p}{p+1}  
      +O \left( X^{\frac{1}{2}+\varepsilon} Y^k \right) \\
      & \qquad +O \left( X^{\frac{1}{2}} (\log Y) \sum_{n_1,\dots,n_k \leq Y} (n_1\cdots n_{k})^\frac{1}{4}\right) \\
    &=\frac{X}{\zeta(2)}\sum_{n_1,\dots,n_k \leq Y}  \prod_{\substack{p \mid n_1\cdots n_{k} \\ n_1 \cdots n_{k}=\square}} \frac{p}{p+1} 
     +O \left( X^{\frac{1}{2}+\varepsilon}Y^{k}\right) +O \left( X^\frac{1}{2}Y^\frac{5k}{4}(\log Y)\right).
\end{split}
\end{align}
From Lemma \ref{count_squares} and Lemma \ref{leading_coefficient} with $\bm{\beta}=\bm{1}$, then we find that $\mathcal{A}_{\bm{\beta}}= \mathcal{A}_k$ and 
\begin{align*}
    \sum_{n_1 \leq Y} \dots \sum_{n_{k} \leq Y} \prod_{\substack{p \mid n_1\cdots n_{k} \\ n_1 \cdots n_{k}=\square}} \frac{p}{p+1} = c_k \gamma_k Y^\frac{k}{2}(\log Y)^{\frac{k(k-1)}{2}}+O\left( Y^\frac{k}{2}(\log Y)^{\frac{k(k-1)}{2}-1} \right).
\end{align*}
Comparing error terms in (\ref{expansion}), the asymptotic formula holds when $Y \ll X^{\frac{2}{3k}}(\log X)^{\frac{2}{3}(k-1)-\frac{4}{3k}}.$
\section{Proofs of the upper bounds}

We will need the following simple lemma.

\begin{lemma}\label{intreal}
Let $X$ be a large parameter, then 
$$I(X):= \int_{1/\log X}^{10}\int_{1/\log X}^{10} (xy)^{-3/4}\frac{1}{\sqrt{y-x}\sqrt{x+y}} \mathds{1}_{y-x \geq 1/\log X} dxdy\ll (\log X)^{1/2}.$$
\end{lemma}
\begin{proof}
By making the transformation $u=y-x$ we get
   \begin{align*} I(X) & \ll  \int_{1/\log X}^{10} x^{-5/4} \left(\int_{1/\log X}^{10} \frac{1}{\sqrt{u}}\frac{1}{(u+x)^{3/4}} du \right) dx \\
   & \ll \int_{1/\log X}^{10} x^{-3/2}  dx  \ll \sqrt{\log X}. \end{align*} 
   \end{proof}

\subsection{Proof of Theorem \ref{th_smooth}} 
Recall the definition of the smoothed moments
$$S_k(X,Y,W)= \sideset{}{^*}{\sum}_{0<d \leq X} \left\vert \sum_{ n\geq 1} \chi_d(n)W(n/Y) \right\vert^k $$
where $W$ is any non-negative, smooth function compactly supported on the set of positive real numbers. 
We follow the initial manipulations performed in \cite{GZ} and arrive to the bound
\begin{equation}\label{Skreduced} S_k(X,Y,W) \ll Y^{k/2} \sideset{}{^*}{\sum}_{0<d \leq X}
   \Big | \int\limits_{\substack{ (1/2) \\ |t| \leq X^{\varepsilon}}}\Big |L(1/2+it, \chi_d) \widehat{W}(1/2+it)\Big|  dt \Big|^{k} \end{equation}
where $\widehat{W}$ is the Mellin transform of $W$ and verifies for any integer $A \geq 0$
\begin{equation}\label{decayW} \widehat{W}(s) \ll \frac{1}{(1+\vert s\vert)^A}.\end{equation} 
 The following result is our main improvement over \cite[Proposition $5.1$]{GZ}. 
\begin{proposition}
\label{bound_intshift}
 Under the assumption of GRH, we have for any fixed integer $k \geq 2$ and any real $10 \leq E=X^{O(1)}$,
\begin{equation*}
  \sideset{}{^*}{\sum}_{0<d \leq X}  \bigg(\int\limits_{0}^{E}|L(\tfrac{1}{2}+it,\chi_d)| d t\bigg)^{k} 
 \ll  X\big( (\log X)^{\frac{k(k-1)}{2}}E^{k}(\log \log E)^{O_k(1)}.
\end{equation*}
\end{proposition}
 The following Lemma follows from Proposition \ref{bound_intshift} and implies Theorem \ref{th_smooth} by \eqref{Skreduced} .
\begin{lemma}\label{lem_intbound}
Assume the truth of GRH. We have for any integer $k \geq 2$ and any $ \varepsilon>0$, 
\begin{equation*}
\label{int_bound}
\sideset{}{^*}{\sum}_{0<d \leq X}
   \Big | \int\limits_{\substack{ (1/2) \\ |t| \leq X^{\varepsilon}}}\Big |L(1/2+it, \chi_d)\widehat{W}(1/2+it)\Big| dt\Big |^{k} \ll X(\log X)^{\frac{k(k-1)}{2 }}.
\end{equation*}
\end{lemma}
 \begin{proof}
 Our proof closely follows the argument in \cite{SZ1}. Using Minkowski's inequality and H\"{o}lder's inequality, we get for $a=1-1/k+\varepsilon$ that
\begin{align}\label{dyadic}
\begin{split}
 \Big | \int\limits_{ |t| \leq X^{\varepsilon}}  & \Big |L(\tfrac{1}{2}+it, \chi_d)\widehat{W}(1/2+it)\Big| dt \Big |^{k} \ll \Big |\int_0^{X^{\varepsilon}} |L(\tfrac{1}{2}+it,\chi_d)| \Big|\widehat{W}(1/2+it)\Big| dt\Big |^{k} \nonumber \\
 & \leq \bigg(\sum_{n\leq \log X+1} n^{-ak/(k-1)} \bigg)^{k-1}
    \sum_{n\leq  \log X+1} \bigg(n^a\int\limits_{e^{n-1}-1}^{e^{n}-1 } \Big|L(\tfrac{1}{2}+it,\chi_d)  \widehat{W}(1/2+it)\Big| dt\bigg)^{k}   \\
  & \ll \sum_{n\leq  \log X+1} \frac{n^{k-1+\varepsilon} }{e^{10nk} } \bigg( \int\limits_{e^{n-1}-1}^{e^{n}-1 } |L(\tfrac{1}{2}+it,\chi_d) | dt \bigg)^{k}
\end{split}\end{align} where we used \eqref{decayW} in the last step. Proposition \ref{bound_intshift} implies that for any integer $k \geq 2$ and any real number $\varepsilon>0$, 
\begin{align*} \sum_{n\leq  \log X+1} & \frac{n^{k-1+\varepsilon} }{e^{10nk} }\sideset{}{^*}{\sum}_{0<d \leq X}  \bigg( \int_{e^{n-1}-1}^{e^{n}-1 } |L(1/2+it,\chi_d) | dt \bigg)^{k}\\ \ll & X (\log X)^{\frac{k(k-1)}{2}}\sum_{n\leq  \log X+1} \frac{n^{k-1+\varepsilon} (\log n)^{O_k(1)}}{e^{9nk} }  \ll X (\log X)^{\frac{k(k-1)}{2}}\end{align*} which concludes the proof of Lemma \ref{lem_intbound}. \end{proof}
\subsection{Proof of Proposition \ref{bound_intshift}}

 By symmetry, we have
\begin{equation}\label{ind_bound}
 I_{k,E}:= \sideset{}{^*}{\sum}_{0<d \leq X} \bigg(\int\limits_{0}^{E}|L(\tfrac{1}{2}+it,\chi_d)| d t\bigg)^{k} 
      \ll \sideset{}{^*}{\sum}_{0<d \leq X} \int\limits_{\mathcal{A}_E}\prod_{a=1}^k|L(1/2+ it_a, \chi_d)| d \mathbf{t}, \end{equation}
where $\mathcal{A}_E=\{ (t_1,\dots,t_k) \in [0,E]^k: 0 \leq t_1 \leq t_2 \dots \leq t_k\}$. 
Summing over $d$ and applying Theorem \ref{th_shifted} we get 
$$ I_{k,E}\ll \int\limits_{\mathcal{A}_E} X(\log X)^{k/4} \prod_{1\leq i<j\leq k} g(|t_i-t_j|)^{1/2}g(|t_i+t_j|)^{1/2}\prod_{1\leq i\leq k} g(|2t_i|)^{3/4} d \mathbf{t}. $$
Letting \begin{equation}\label{defJ}
J_{k,E}:= (\log X)^{k/4}  \int\limits_{\mathcal{A}_E} \prod_{1\leq i<j\leq k} g(|t_i-t_j|)^{1/2}g(|t_i+t_j|)^{1/2}\prod_{1\leq i\leq k} g(|2t_i|)^{3/4} d \mathbf{t}, \end{equation} we want to show that for all integers $k\geq 2$ \begin{equation}\label{ind_hyp}
J_{k,E} \ll E^k (\log X)^{\frac{k(k-1)}{2}}(\log \log E)^{O_k(1)}
\end{equation} where the implied constant depends only on $k$.
We proceed by induction on the number of variables $k$.
\subsubsection{The base case $k=2$} 
We make repeated use of the bounds \eqref{gDef} on the function $g$.
We first remark that if $t_2 \geq t_1 \geq 10$ we can trivially bound every term in \eqref{defJ} by $(\log \log E)$ except one term where we use $g(t_2-t_1) \leq \log X.$ Hence the contribution to $J_{2,E}$ is at most $\ll E^2 (\log X) (\log \log E)^{2}$. If $t_1 \leq 1/\log X$ and $ t_2 \leq 2/\log X$, a trivial volume argument bounds the contribution by $\ll \log X.$  A similar bound holds in the case  $t_1 \leq 1/\log X$ and $t_2 \geq 2/\log X $. For $1/\log X \leq t_1 \leq 10 \leq t_2$, we get a contribution  $\ll E (\log X)(\log \log E)^{2}$.
Let us turn to the most problematic case  $1/\log X \leq t_1, t_2 \leq 10$. If $t_2-t_1 \leq 1/\log X$, we use the bound $g(t_2-t_1) \ll (\log X)$ and get a contribution to $J_{2,E} $ which is
\begin{align*}  & \ll (\log X)  \int_{1/\log X}^{10}  \frac{1}{t_1^2}  \int_{t_1 \leq t_2 \leq t_1+1/\log X} dt_2 dt_1  \ll \log X.
\end{align*} In the remaining subcase, the result follows directly from Lemma \ref{intreal}.
\subsubsection{The induction step}
We now assume that the following bound holds:
\begin{equation}\label{boundJk1} J_{k-1,E} \ll E^{k-1}(\log X)^{\frac{(k-1)(k-2)}{2}}(\log \log E)^{O_k(1)}. \end{equation}
We will use our induction hypothesis to bound the integral over $t_2,\dots,t_k$ and use a pointwise bound for the remaining factors involving the variable $t_1$. To do so, for any $(t_1,\dots,t_k) \in [0,E]^k$, we let 
$$F(t_1,\dots,t_k):= g(|2t_1|)^{3/4}\prod_{2\leq j\leq k} g(|t_1-t_j|)^{1/2}g(|t_1+t_j|)^{1/2}.$$ We have the bound $F(t_1,\dots,t_k) \ll h_{k,E}(t_1)$ where the function $h:=h_{k,E}$ verifies 
$$ h_{k,E}(t_1)= \begin{cases} (\log X)^{k-1/4}, & t_1 \leq 1/\log X  \\
t_1^{-(k-1)/2-3/4}(\log X)^{(k-1)/2} (\log \log E)^{k-1}, & 1/\log X \leq t_1 \leq 10  \\
(\log X)^{(k-1)/2}(\log \log E)^{(k-1)/2+3/4}, & 10\leq t_1 \leq E.
\end{cases}$$ Indeed, in the first case we trivially used the boud $g(t) \ll \log X$ for each term appearing in the product. In the second case, we notice that $t_1 + t_j \geq t_1$ for any $j=2,\dots,k$. Hence, $g(t_1+t_j) \ll g(t_1) \log \log E \ll \frac{1}{t_1^{1/2}} (\log \log E)^2$. We bound trivially $g(t_1-t_j) \ll \log X$ for $j=2,\dots,k$. In the last case, we use the bound $g(t_1+t_j), g(t_1) \ll \log \log E$ for $j=2,\dots,k$ and $g(t_1-t_j) \ll \log X$ for $j=2,\dots,k$. By \eqref{defJ} we have
$$ \frac{J_{k,E}}{(\log X)^{\frac{1}{4}}} \ll \int_{0}^{E} h(t_1) \left\{(\log X)^{\frac{k-1}{4}} \int\limits_{\mathcal{B}_E} \prod_{2\leq i<j\leq k} g(|t_i-t_j|)^{1/2}g(|t_i+t_j|)^{1/2} \prod_{2\leq i \leq k} g(|2t_i|)^{3/4}  d \mathbf{t} \right\} dt_1 $$
where $\mathcal{B}_E=\{ (t_2,\dots,t_k) \in [0,E]^{k-1}: 0 \leq t_2 \leq t_3 \dots \leq t_k\}$. Thus, by our induction hypothesis,
\begin{align}\label{boundJk} J_{k,E} &\ll J_{k-1,E} \,\,(\log X)^{1/4} \int_{0}^{E} h(t_1)dt_1 \nonumber\\
& \ll E^{k-1}(\log \log E)^{O_k(1)} (\log X)^{\frac{(k-1)(k-2)}{2}+1/4} \int_{0}^{E} h(t_1)dt_1 .\end{align} 
A simple computation shows that 
\begin{align}\label{boundh} \int_{0}^{E} h(t_1)dt_1 & =  \int_{0}^{1/\log X} h(t_1)dt_1 + \int_{1/\log X}^{10} h(t_1)dt_1 +\int_{10}^E h(t_1)dt_1 \nonumber \\
 &\ll (\log X)^{k-5/4} (\log \log E)^{k-1} +(\log X)^{(k-1)/2} E(\log \log E)^{(k-1)/2+3/4}.\end{align}
Hence, by \eqref{boundJk} and \eqref{boundh},
\begin{align*} J_{k,E}& \ll E^{k-1}(\log X)^{\frac{k(k-1)}{2}}(\log \log E)^{k-1+O_k(1)} + (\log X)^{\frac{(k-1)^2}{2}}E^k(\log \log E)^{k/2+1/4+O_k(1)} \\
& \ll (\log X)^{\frac{k(k-1)}{2}}E^k(\log \log E)^{O_k(1)}. \end{align*}
This concludes the proof of the Proposition.

     \subsection{Proof of Theorem \ref{th_theta}}
We omit some details as the proof is completely similar to the proof of Theorem \ref{th_smooth}. Indeed, for every $d>0$ and the associated even primitive character $\chi_d$, we have for $c>1/2$
$$\theta(1,\chi_d)=\int_{c-i\infty}^{c+\infty} L(2s,\chi_d)\left(\frac{d}{2\pi}\right)^{s}\Gamma(2s)ds.$$ Shifting the line of integration to $\Re(s)=1/4$ and using the decay of $\Gamma(s)$ in vertical lines, we end up with 
$$\theta(1,\chi_d)=\left(\frac{d}{\pi}\right)^{\frac{1}{4}}\int_{-\infty}^{\infty}L\left(\frac{1}{2}+2it,\chi_d\right)\left(\frac{d}{\pi}\right)^{2it}\Gamma\left(\frac{1}{2}+2it\right)dt.$$
Hence, we obtain
\begin{equation}\label{moments2k}\sideset{}{^*}{\sum}_{0<d\leq X} \vert \theta(1,\chi_d)\vert^{k}\ll X^{\frac{k}{4}} \sideset{}{^*}{\sum}_{0<d\leq X}\left\vert \int_{-\infty}^{\infty}L\left(\frac{1}{2}+2it,\chi_d\right)\left(\frac{d}{\pi}\right)^{2it}\Gamma\left(\frac{1}{2}+2it\right)dt\right\vert^{k}.\end{equation}  By Stirling's formula, the Gamma function decays faster than any polynomial on vertical lines. For instance, we have \begin{equation*} \Gamma(1/2+it) \ll \frac{1}{(1+\vert t\vert)^{10}}. \end{equation*} Thus, the tails of the integral are easily seen to be negligible, and we are left to bound \begin{equation*}\sideset{}{^*}{\sum}_{0<d\leq X}
   \left( \int\limits_{\substack{ (1/2) \\ |t| \leq X^{\varepsilon}}}\left| L(1/2+2it, \chi_d) \Gamma\left(\frac{1}{2}+2it\right) \right| dt \right)^{k} \end{equation*}    We apply Proposition \ref{bound_intshift} to obtain an analogue of Lemma \ref{lem_intbound} with $\widehat{W}$ replaced by $\Gamma$. This implies Theorem \ref{th_theta}.

     \section{The second moment of $\theta(1,\chi_d)$: Proof of Theorem \ref{theta2nd}}
The lower bound of Theorem \ref{theta2nd} follows from Theorem \ref{lowtheta} for $k=2$. We now switch to the upper bound. We have 
$$  \sideset{}{^*}{\sum}_{0<d \leq X} \vert \theta(1,\chi_d)\vert^2 \ll \sum_{d \in \mathcal{D}(X)} \vert \theta(1,\chi_d)\vert^2 $$ where $\mathcal{D}(X)$ denotes the set of quadratic discriminants $d \leq X$. Letting $\chi_d(n)=1$ when $d$ is a square, we can include the squares in the summation. This does not add more to the sum than $\sum_{k \leq X^{1/2}} \sum_{n=1}^{+\infty} e^{-\pi n^2/k^2} \ll X$. Therefore, we have 
\begin{align*} \sideset{}{^*}{\sum}_{0<d \leq X} \vert \theta(1,\chi_d)\vert^2  &\ll \sum_{m,n} \sum_{0<d \leq X} \chi_d(mn) e^{-\pi (m^2+n^2)/d} =T_2^{sq}+T_2^{nsq}  \end{align*}
where $T_2^{sq}$ denotes the summation over $m,n$ when $mn$ is a square and $T_2^{nsq}$ denotes its counterpart. For fixed $m,n$, 
proceeding \footnote{The only difference comes from the restriction over squarefree $d$'s in \cite{LM} which needs minor modifications.} as in \cite[Lemma $21$]{LM}, we see that $T_2^{sq} \asymp X^{3/2} \log X$ (an explicit asymptotic formula could be proved). It remains to deal with the non-square terms $mn$. In this case there exists a non-principal character $\chi_{mn}$ mod $mn$ or $4mn$ such that $\chi_{mn}(d)=\chi_d(mn)$. We need to bound
$$\sum_{m,n \atop mn \neq \square} \sum_{0<d\leq X} \chi_{mn}(d)e^{-\pi(m^2+n^2)/d}.$$ We first split the summation over $d$ in dyadic intervals  $(X/2^{k+1},X/2^k]$ with $0 \leq k \ll \log X$. We apply Lemma \ref{lem:weight} to the sum over $d$ with $b_d=e^{-\pi (m^2+n^2)/d}$. Clearly $b_d \leq e^{-\pi (m^2+n^2)2^{k-1}/X}$ for any fixed $m,n$ and $d$ in the dyadic range. Moreover, for fixed $m,n$, we have by the mean value inequality 
\begin{align*} V&=\sum_{X/(2^{k+1})<d\leq X/2^k} \left\vert b_{d}-b_{d+1} \right\vert \ll \sum_{X/(2^{k+1})<d\leq X/2^k} \frac{(m^2+n^2)}{d^2} e^{-\pi (m^2+n^2)/d} \\
& \ll  e^{-\pi (m^2+n^2)2^{k-1}/X} \sum_{X/(2^{k+1})<d\leq X/2^k} \frac{1}{d}  \ll e^{-\pi (m^2+n^2)2^{k-1}/X}. \end{align*}

Hence, using P\'{o}lya-Vinogradov inequality and Lemma \ref{lem:weight} we get 
\begin{align*} T_2^{nsq}& \ll \sum_k \sum_{m,n} m^{1/2}n^{1/2} \log(mn) e^{-\pi (m^2+n^2)2^{k-1}/X}  \\
& \ll \sum_k \left( \sum_m m^{1/2} e^{-\pi m^2 2^{k-1}/X} \right)\left( \sum_n n^{1/2} (\log n) e^{-\pi n^2 2^{k-1}/X} \right) \\
& \ll \sum_k (X/2^k)^{3/2} \log X \ll X^{3/2} \log X \end{align*} which concludes the proof.

\section*{Acknowledgement}

During the preparation of this work, M. M. was supported by AAP Recherche 2025 UJM ``Comportements al{\'e}atoires en arithm{\'e}tique". Y. T. would like to thank Dr. Dilip Kumar Sahoo and Professor Kohji Matsumoto for their valuable comments. Some parts of this research were carried out when Y. T. stayed at IISER Berhampur in India from June to July in 2024. Y. T. would like to thank Professor Kasi Viswanadham Gopajosyula
and the staff of IISER Berhampur for their support and hospitality. Y. T. is supported by Grant-in-Aid for JSPS Research Fellow (Grant Number:24KJ1235).

\Addresses
\end{document}